\documentclass{amsart}
\pdfoutput=1
\usepackage[numbers]{natbib}

\usepackage[utf8]{inputenc}
\usepackage[mathscr]{eucal}
\usepackage{amssymb}
\usepackage[usenames,dvipsnames]{xcolor}
\usepackage{xspace}
\usepackage{amsthm}
\usepackage{bbold}
\usepackage[normalem]{ulem}
\usepackage{enumitem}
\usepackage{array}
\usepackage{amsmath}
\usepackage{wasysym}
\usepackage{etoolbox}

\usepackage{comment}

\numberwithin{equation}{section}
\usepackage{tikz-cd}

\usepackage[all]{xy}
\SelectTips{cm}{}
\newdir{ >}{{}*!/-10pt/\dir{>}}

\usepackage[colorlinks=true,linkcolor={Brown},citecolor={Brown},urlcolor={Brown}]{hyperref}

\usepackage{cleveref}

\newcounter{thm-intro}
\newtheorem{Thm*}[thm-intro]{Theorem}

\newtheorem{Lem}[equation]{Lemma}

\newtheorem{Prop}[equation]{Proposition}
\newtheorem{Thm}[equation]{Theorem}

\theoremstyle{remark}
\newtheorem{Def}[equation]{Definition}

\newtheorem{Hyp}[equation]{Hypothesis}
\newtheorem{Rem}[equation]{Remark}

\newcommand{\nc}{\newcommand}
\nc{\dmo}{\DeclareMathOperator}

\dmo{\car}{char}
\dmo{\coker}{coker}
\dmo{\cone}{cone}
\dmo{\Der}{D}
\dmo{\Ext}{Ext}
\dmo{\Gal}{Gal}
\dmo{\Hm}{H}
\dmo{\Hom}{Hom}
\dmo{\Ind}{Ind}
\dmo{\modname}{mod}%
\dmo{\Mod}{Mod}
\dmo{\Spc}{Spc}
\dmo{\Spec}{Spec}
\dmo{\Spech}{Spec^h}
\dmo{\stab}{stab}
\dmo{\supp}{supp}
\dmo{\uAut}{\underline{Aut}^{\otimes}_{\KK}}

\nc{\bti}{\mathrm{Re}_{\mathrm{B}}}
\nc{\cat}[1]{\mathscr{#1}}
\nc{\cA}{\cat{A}}
\nc{\cB}{\cat{B}}
\nc{\colim}{\mathop{\mathrm{colim}}}
\nc{\cP}{\cat{P}}
\nc{\cT}{\cat{T}}
\nc{\cX}{\mathscr{X}}
\nc{\ladic}{\mathrm{Re}_{\ell}}
\nc{\eg}{{\sl e.g.}\xspace}
\nc{\holim}{\mathop{\mathrm{holim}}}
\nc{\ie}{{\sl i.e.}\ }
\nc{\into}{\mathop{\rightarrowtail}}
\nc{\inv}{^{-1}}
\nc{\KM}{K^{\mathrm{M}}}
\nc{\loccit}{{\sl loc.\ cit.}\xspace}
\nc{\nori}{\mathrm{N}} 
\nc{\onto}{\mathop{\twoheadrightarrow}}
\nc{\op}{^{\opname}}
\nc{\potimes}[1]{^{\otimes #1}}
\nc{\sbull}{{\scriptscriptstyle\bullet}}
\nc{\unit}{\mathbb{1}}

\dmo{\End}{End}
\nc{\isoto}{\overset{\sim}{\,\to\,}}
\dmo{\chr}{char}%
\dmo{\comod}{comod}
\dmo{\DMcet}{DM^{\mathrm{c}}_{\textup{\et}}}
\dmo{\DMlet}{DM^{\vee}_{\textup{\et}}}
\dmo{\DMch}{DM_{\mathrm{h,c}}}
\dmo{\DMlch}{DM_{\mathrm{h,lc}}}
\dmo{\DMgmet}{DM^{\mathrm{gm}}_{\textup{\et}}}
\dmo{\DMbig}{DM}
\dmo{\DM}{DM^{\geom}}
\dmo{\DMbiget}{DM_{\textup{\et}}}
\dmo{\HM}{HM} 
\dmo{\HMeff}{HM^{\mathrm{eff}}}
\dmo{\hoComod}{hoComod}
\dmo{\id}{id}
\dmo{\Img}{Im}
\dmo{\Komp}{K}
\dmo{\modules}{mod}%
\dmo{\filmod}{mod^{fil}}
\dmo{\MM}{MM}
\dmo{\smallb}{b}
\dmo{\smallc}{c}
\dmo{\geom}{gm}

\nc{\cf}{{\sl cf.}\ }
\nc{\Db}{\Der^{\smallb}}
\nc{\Dbc}{\Der^{\smallb}_{\smallc}}
\nc{\D}{\Der}

\nc{\et}{\textrm{\'et}}
\nc{\etal}{{\sl et al.}}

\nc{\Kb}{\Komp^{\smallb}}
\dmo{\VBName}{VB}
\nc{\VB}[1]{\VBName_{#1}}
\dmo{\QCName}{QC}
\nc{\QC}[1]{\QCName_{#1}}

\newcommand{\CC}{\mathbb{C}}
\newcommand{\FF}{\mathbb{F}} 
\newcommand{\GG}{\mathbb{G}}
\newcommand{\KK}{\mathbb{K}}
\newcommand{\QQ}{\mathbb{Q}}
\newcommand{\ZZ}{\mathbb{Z}}

\newcommand{\kw}{motives, \'etale motives, tensor triangular geometry, tannakian category, classification}
\AtEndPreamble{\hypersetup{
    pdfcreator    = {\LaTeX{}},
    pdfencoding   = auto,
    psdextra,
    pdfauthor     = {Martin Gallauer},
    pdftitle      = {A note on Tannakian categories and mixed motives},
    pdfsubject    = {},
    pdfkeywords   = {\kw}}}

\begin{document}
\title{A note on Tannakian categories and mixed motives}
\author{Martin Gallauer}
\address{Martin Gallauer, Mathematical Institute, University of Oxford}
\email{gallauer@maths.ox.ac.uk}
\urladdr{https://people.maths.ox.ac.uk/gallauer}
\begin{abstract}
  We explain why every non-trivial exact tensor functor on the triangulated category of mixed motives over a field $\FF$ has zero kernel, if one assumes ``all'' motivic conjectures.
  In other words, every non-zero motive generates the whole category up to the tensor triangulated structure.
  Under the same assumptions, we also give a complete classification of triangulated \'etale motives over $\FF$ with integral coefficients, up to the tensor triangulated structure, in terms of the characteristic and the orderings of $\FF$.
\end{abstract}

\maketitle

\section{Introduction}
\label{sec:intro}
Let $\FF$ be a field, and consider Voevodsky's triangulated category of motives over~$\FF$ with rational coefficients, as usual denoted by $\DM(\FF;\QQ)$.
It is a strong candidate for the derived category of a conjectural abelian category of mixed motives over $\FF$:
\begin{equation}
  \label{eq:DM=Db}
  \DM(\FF;\QQ)\overset{?}{\simeq}\Db(\MM(\FF;\QQ))
\end{equation}
However, at this point, the construction of $\MM(\FF;\QQ)$ and the equivalence~(\ref{eq:DM=Db}) seem out-of-reach, and the global structure of $\DM(\FF;\QQ)$ remains mysterious.
In the present article we will instead explore some consequences of such an equivalence.
More specifically, we will classify motives in different contexts assuming that~(\ref{eq:DM=Db}) holds.
(We refer to the discussion around~\Cref{big-conjecture} for the precise assumptions.)
Here is the first result in that direction (see \Cref{spc-rational}).
\begin{Thm*}
  \label{ThmA}
  Assume~(\ref{eq:DM=Db}). Then the tensor triangulated category $\DM(\FF;\QQ)$ is simple.
\end{Thm*}
\noindent Recall that a rigid tensor triangulated category is called simple if 0 is the only proper thick tensor ideal.
For example, if $\mathrm{Re}:\DM(\FF;\QQ)\to\Db(K)$ denotes some classical (co)homology theory, its kernel is a proper thick tensor ideal and therefore necessarily 0, by this result.
This consequence is known as the \emph{Conservativity Conjecture} (for the given (co)homology theory). 
\Cref{ThmA} says that \emph{every} non-trivial tensor triangulated functor $F:\DM(\FF;\QQ)\to\cT$ is conservative.
Another way of thinking about it is that starting with a non-zero motive $M$ in $\DM(\FF;\QQ)$, \emph{every} other motive can be constructed from $M$ using shifts, cones, direct summands, and tensoring with other motives.

It is expected that $\MM(\FF;\QQ)$ is a Tannakian category (not necessarily neutral), and this is part of our assumptions in \Cref{ThmA}.
Hence the latter will follow from the following one which we learned from David Rydh (see \Cref{tannaka-classification}).
\begin{Thm*}
  \label{ThmB}
  Let $\cA$ be a Tannakian category in characteristic zero.
  Then its bounded derived category $\Db(\cA)$ is simple.
\end{Thm*}
\noindent This can be seen as a derived analogue of the easy observation that a Tannakian category in characteristic zero is simple.\footnote{In other words, 0 is the only proper thick tensor ideal, and where \emph{thick tensor ideals} can be interpreted in any reasonable sense. Whereas, in positive characteristic, a Tannakian category is simple only if thick tensor ideals are interpreted in a strong sense, \cf~\Cref{sec:tannaka}.}
As Tannakian categories appear frequently in algebraic geometry (and elsewhere), this result is of independent interest.
The proof of \Cref{ThmB} proceeds by translating the problem into one about quasi-coherent sheaves on well-behaved stacks where one can exploit finite cohomological dimension in characteristic zero~\cite{MR3037900,hall-rydh:compact-generation,hall-rydh:telescope-stacks}.

So far, our discussion has ignored any motivic torsion phenomena, and the first steps in the study of $\DM(\FF;R)$ for finite coefficients $R$ (as in~\cite{balmer-gallauer:tt-ratm} and~\cite{balmer-gallauer:artin-motives}) suggest a complicated picture indeed.
However, the situation becomes better when passing to \emph{\'etale} motives, as was observed in~\cite{gallauer:tt-dtm-algclosed}. 
Here we improve on that result, including an argument by Paul Balmer which allows to remove assumptions on roots of unity.
At the end of the day, we are able to give a complete classification of thick tensor ideals in the triangulated category of mixed \'etale motives over $\FF$ with integral coefficients, under the same assumptions as before.
The result is best expressed in terms of Balmer's tensor triangular spectrum (that is, the space of prime ideals) as follows (see \Cref{spc-etale-motives}).
\begin{Thm*}
  Let $\FF$ be a field of exponential characteristic $p$ and assume that $\DM(\FF;\QQ)$ is simple.
  The comparison map with respect to the Tate motive $\tilde{\GG}_m=\ZZ(1)[1]$ is a homeomorphism
  \begin{equation*}
    \rho^{\sbull}_{\tilde{\GG}_m}:\Spc(\DMgmet(\FF;\ZZ))\isoto\Spech(\KM_{\sbull}(\FF)[1/p])
  \end{equation*}
  onto the space (described in \Cref{structure-map-milnor}):
  \begin{equation}
    \label{eq:space}%
    \vcenter{%
      \xymatrix{
      & & & \mathfrak{m}_2  \ar@{-}[ld] \ar@{-}[rd] \ar@{-}[rrrd] & &\mathfrak{m}_3 \ar@{-}[rd] &\mathfrak{m}_5 \ar@{-}[d]&\cdots&\mathfrak{m}_\ell \\
      & &\mathfrak{p}_\alpha& ... &\mathfrak{p}_\beta& &\mathfrak{t}\ar@{-}[rur]\ar@{}[ru]|{\cdots}}
  }
\end{equation}
\end{Thm*}
Here, $\KM_{\sbull}(\FF)[1/p]$ is the graded Milnor K-theory algebra localized away from $p$.
The $\mathfrak{p}_\sbull$ are indexed by the orderings of $\FF$ with the Harrison topology, and $\ell$ runs through the rational primes different from $p$.
The lines indicate specialization relations pointing upward.
In particular, if $\FF$ does not admit any orderings (for example in positive characteristic) then the tensor triangular spectrum is just $\Spec(\ZZ[1/p])$.
We refer to \Cref{primes} for a description of the prime ideals in $\DMgmet(\FF;\ZZ)$.
We also recall that the thick tensor ideals then correspond to the Thomason subsets of the space in~(\ref{eq:space}) (that is, the unions of closed subsets with quasi-compact complements)~\cite[Theorem~4.10]{balmer:spectrum}.
In other words, up to the tensor triangulated structure, \'etale motives are classified by the Thomason subsets in~(\ref{eq:space}).

\subsection*{Acknowledgments}

We thank Joseph Ayoub, Paul Balmer, and David Rydh for instructive discussions which helped shape this article.
Thanks also to Denis-Charles Cisinski for valuable feedback on an earlier version, and to the anonymous referee for his/her suggestions for improving the exposition.

\section{Simplicity of (derived) Tannakian categories}
\label{sec:tannaka}
Let $\cA$ be a rigid exact $\otimes$-category.
(In other words, $\cA$ is endowed with an exact and a symmetric monoidal structure, the tensor product is exact in each variable separately, and each object is rigid, that is, has a tensor dual.)
A \emph{thick tensor ideal} of $\cA$ is a non-empty strictly full exact subcategory closed under retracts, extensions, and under tensoring with arbitrary objects in $\cA$.
Examples are kernels of exact tensor functors with domain $\cA$.
Let us then call $\cA$ \emph{simple} if $0$ is the only proper thick tensor ideal.
Put differently, $\cA$ is not the zero category, and the thick tensor ideal generated by any non-zero object contains the unit.

\begin{Rem}[Abelian tensor categories]
  When $\cA$ is an \emph{abelian} tensor category, stronger notions of ideals are equally possible, for example by requiring them to be, in addition, abelian or Serre subcategories.
  The former are called \emph{coherent tensor ideals} in~\cite{peter:spectrum-damt}, and the latter are called \emph{Serre (tensor) ideals} in~\cite{balmer-krause-stevenson:frame-smashing}.
  Again, any kernel of an exact tensor functor is an example of such subcategories.
  The notion introduced above is closer to the one used in tensor triangulated categories which we will focus on below.
\end{Rem}

The following basic observation is the origin of the results discussed in the present section.
Recall that for a field $\KK$, a rigid abelian tensor category with $\End(\unit)=\KK$ is called \emph{Tannakian} if it admits a fiber (= faithful exact $\KK$-linear tensor) functor with values in modules over a non-zero $\KK$-algebra.
The characteristic of the Tannakian category is the characteristic of $\KK$.
\begin{Lem}
  \label{tannakian-char0-simple}
  Let $\cA$ be a Tannakian category in characteristic zero.
  Then $\cA$ is simple.
\end{Lem}
\begin{proof}
  The category $\cA$ is the filtered union of its algebraic Tannakian subcategories (that is, those generated by a single object).
  It is therefore sufficient to prove the statement for $\cA$ algebraic.
  In that case, there is a fiber functor $\omega:\cA\to\modules(\KK')$ with $\KK'/\KK$ a finite field extension~~\cite[Corollaire~6.20]{deligne:cat-tann}.

  Let $a\in\cA$ be a non-zero object.
  Then $\omega(a)$ is a $\KK'$-vector space of finite dimension, $d>0$, say.
  We claim that $\Lambda^da\otimes
  \left(
    \Lambda^da
  \right)^*$ belongs to the thick tensor ideal generated by $a$.
  (Here, $b^*$ denotes the tensor dual of $b$.)
  Indeed, it suffices to note that $a\potimes{d}$ does, and that $\Lambda^da$ is a direct summand of $a\potimes{d}$ since we are in characteristic zero.
  But since the canonical morphism $\Lambda^da\otimes
  \left(
    \Lambda^da
  \right)^*\to\unit$ becomes invertible after applying $\omega$, it must be invertible itself.
  We conclude that $\unit$ belongs to the thick tensor ideal generated by $a$ as required.
\end{proof}

\begin{Rem}[Positive characteristic]
  In particular, every exact tensor functor $F:\cA\to\cB$ with $\cA$ a Tannakian category in characteristic zero, and $\cB\neq 0$, is faithful.
  In fact, this latter conclusion is true if $\cA$ is a Tannakian category in positive characteristic as well.
  Indeed, the proof of the Lemma shows that the only Serre tensor ideals of $\cA$ are $0$ and $\cA$.

  On the other hand, it is not true in general that a Tannakian category in positive characteristic is simple in the sense introduced at the beginning of this section.
  For example, for a finite group $G$, the category of $G$-representations over a field of characteristic dividing the order of $G$ contains a non-trivial thick tensor ideal consisting of the projective objects.
\end{Rem}

The goal in the remainder of this section is to prove the following derived analogue of \Cref{tannakian-char0-simple}.

\begin{Thm}
  \label{tannaka-classification}%
  Let $\cA$ be a Tannakian category in characteristic zero.
  Then $\Db(\cA)$ is simple.
\end{Thm}

Here, a \emph{thick tensor ideal} of a tensor triangulated category is a non-empty full subcategory closed under retracts, shifts, cones, and tensoring with any object in the category.
Again, a rigid tensor triangulated category is \emph{simple} if the only proper thick tensor ideal is the zero ideal.

\begin{Lem}
  In proving \Cref{tannaka-classification} we may assume that $\cA$ is an \emph{algebraic} Tannakian category, that is, $\cA$ is generated (as a Tannakian subcategory) by a single object.
\end{Lem}
\begin{proof}
  The category $\cA$ is the filtered union of algebraic Tannakian subcategories $\cA_i$.
  It then follows from \Cref{Db-filtered-colimit} below that $\Db(\cA)$ is the filtered colimit of the $\Db(\cA_i)$.
  The filtered colimit of simple tensor triangulated categories is clearly simple (or one may use~\cite[Proposition~8.5]{gallauer:tt-fmod}).
\end{proof}
\begin{Lem}
  \label{Db-filtered-colimit}%
  Let $\cA$ be the filtered colimit of exact categories $\cA_i$ (with exact functors between them). 
  This induces an equivalence of triangulated categories
  \begin{equation}
    \label{eq:filtered-colimit-Db}
    \varinjlim_i\Db(\cA_i)\isoto\Db(\cA).\footnote{This direct limit can be interpreted as a pseudo-colimit in the 2-category of small triangulated categories, exact functors, and exact isotransformations. But the content of this is that there is an evident way of endowing the filtered colimit of categories with a triangulated structure if each category is triangulated and each functor is exact. Cf.~\cite[Remark~8.3]{gallauer:tt-fmod}}
  \end{equation}
\end{Lem}
\begin{proof}
  We first treat the case where all $\cA_i$ (and therefore also $\cA$) are endowed with the \emph{split} exact structure.
  In other words, we want to prove the statement for the bounded homotopy categories:
  \begin{equation}
    \varinjlim_i\Kb(\cA_i)\isoto\Kb(\cA)\label{eq:filtered-colimit-Kb}
  \end{equation}
  The finitely many objects and morphisms making up a bounded complex in $\cA$ all live in some $\cA_i$, and define an object of $\Kb(\cA_i)$.
  Thus essential surjectivity of~(\ref{eq:filtered-colimit-Kb}).
  Now fix two bounded complexes $X,Y\in\Kb(\cA_{i_0})$.
  Given a morphism $f:X\to Y$ in $\Kb(\cA)$, its components $f_n:X_n\to Y_n$ must lie in some $\cA_{i_1}$ for $i_1\geq i_0$.
  The fact that these components define a morphism of complexes in $\cA$ translates into an identity of certain morphisms.
  There are finitely many such identities, and therefore these hold already in some $\cA_{i_2}$ with $i_2\geq i_1$.
  This proves fullness of~(\ref{eq:filtered-colimit-Kb}).
  With $X$ and $Y$ as above, let us be given a morphism  $f:X\to Y$ in $\Kb(\cA_{i_0})$ which becomes zero in $\Kb(\cA)$.
  In other words, there is a homotopy between $f$ and the zero morphism.
  Again, this homotopy consists of finitely many morphisms satisfying finitely many identities and therefore already exists in some $\cA_{i_1}$ with $i_1\geq i_0$.
  This proves faithfulness of~(\ref{eq:filtered-colimit-Kb}).

  Recall that $\Db$ is a localization of $\Kb$.
  Essential surjectivity of~(\ref{eq:filtered-colimit-Db}) therefore follows immediately from the split exact case.
  Now fix $X,Y\in\Db(\cA_{i_0})$.
  A morphism $X\to Y$ in $\Db(\cA)$ consists of a roof $X\leftarrow X'\to Y$ in $\Kb(\cA)$ where the first arrow is a quasi-isomorphism (that is, its cone is isomorphic to an acyclic complex~\cite[\S\,11]{keller:derived-categories}).
  By the previous case we already know that this roof exists in some $\Kb(\cA_{i_1})$ with $i_1\geq i_0$.
  Let $C$ be the cone of the first arrow in $\Kb(\cA_{i_1})$.
  We know that $C$ becomes isomorphic to an acyclic complex $C'$ in $\Kb(\cA)$ so $C'$ and this isomorphism already exist in some $\Kb(\cA_{i_2})$ with $i_2\geq i_1$.
  Finally, acyclicity of $C'$ already holds in some $\Kb(\cA_{i_3})$ with $i_3\geq i_2$, and we deduce that~(\ref{eq:filtered-colimit-Db}) is full.
  We will complete the proof by showing that~(\ref{eq:filtered-colimit-Db}) is conservative.
  (Recall that a full exact functor between triangulated categories is faithful if and only if it is conservative.)
  For this, let $X\in\Kb(\cA_{i_0})$ be a complex which becomes isomorphic to an acyclic complex $X'$ in $\Kb(\cA)$.
  As in the proof of fullness, this already happens in some $\Kb(\cA_{i_1})$ with $i_1\geq i_0$, and $X$ is already 0 in the domain of~(\ref{eq:filtered-colimit-Db}).
\end{proof}

For the proof of \Cref{tannaka-classification} we will therefore assume that $\cA$ is algebraic.
Let $\KK=\End_{\cA}(\unit)$ denote the field of characteristic zero over which $\cA$ is defined.
As in the proof of \Cref{tannakian-char0-simple}, algebraicity forces the existence of a fiber functor $\omega:\cA\to\modules(\KK')$ over a finite extension $\KK'/\KK$~\cite[Corollaire~6.20]{deligne:cat-tann}.
\begin{proof}[Proof of \Cref{tannaka-classification}]
  We continue to denote by $\omega$ the (essentially unique) cocontinuous extension
  \[
    \Ind(\cA)\to\Ind(\modules(\KK'))\simeq\Mod(\KK'),
  \]
  which is a faithful, exact $\otimes$-functor.
  (This is easy to prove directly but also follows from the interpretation in terms of stacks below.)
  It therefore derives trivially to a conservative tensor triangulated functor preserving coproducts
  \[
    \omega:\Der(\Ind(\cA))\to\Der(\KK').
  \]
  We will prove below that $\Der(\Ind(\cA))$ is rigidly compactly generated by $\Db(\cA)$ (that is, $\Der(\Ind(\cA))$ is compactly generated, $\Db(\cA)$ is the full subcategory of compact objects, and these are moreover rigid).
  It is then a formal consequence that $\omega$ induces an injective map from the set of thick tensor ideals in $\Db(\cA)$ to those in $\Db(\KK')$~\cite[Corollary~4.2]{hall-rydh:telescope-stacks}.
  As the latter category is simple so is the former, as required.

  It remains to prove that $\Der(\Ind(\cA))$ is rigidly compactly generated by $\Db(\cA)$, and it will be convenient to do this using the language of algebraic stacks.
  The Tannakian dual $\uAut(\omega)$ is an affine algebraic $\KK$-groupoid acting on $\Spec(\KK')$.
  Let $\cX$ be the algebraic stack (a gerbe) it presents.
Recall that, by~\cite[Th\'eor\`eme~1.12, Corollaire~3.9]{deligne:cat-tann}, $\omega$ induces an equivalence
\[
  \Ind(\cA)\simeq\QC{\cX}
\]
with the category of quasi-coherent sheaves on $\cX$.
The characteristic zero assumption buys us finite cohomological dimension of $\cX$, which translates into compact generation of $\Der(\QC{\cX})$ by the perfect (or equivalently, rigid) objects~\cite[Remark~1.2.10, Theorem~1.4.2, Corollary~1.4.3]{MR3037900} and \cite[Corollary~3.5]{hall-rydh:compact-generation}.
The latter span precisely the image of the fully faithful functor
  \[
    \Db(\cA)\to\Der(\Ind(\cA))\simeq\Der(\QC{\cX})
  \]
  and we conclude.
\end{proof}

\begin{Rem}[Underived and derived simplicity]
  \label{deriving-simplicity}%
  If $\cA$ is a simple exact tensor category, it is not necessarily the case that $\Db(\cA)$ is simple.
  For example, let $\KK$ be a field and consider the category $\filmod(\KK)$ of finite dimensional $\KK$-vector spaces with an exhaustive and separated filtration.
  It is a quasi-abelian category, and therefore an exact category in which the admissible epimorphisms (resp.\ admissible monomorphisms) are the strict surjections (resp.\ strict injections)~\cite[\S\,2]{gallauer:tt-fmod}.
  Together with the tensor product of filtered vector spaces, the exact tensor category $\filmod(\KK)$ is easily seen to be simple.
  On the other hand, it is proven in~\cite[Proposition~7.8]{gallauer:tt-fmod} that its derived category $\Db(\filmod(\KK))$ is not simple.
  Indeed, a (unique) non-trivial proper thick tensor ideal is generated by the cone of the morphism $\beta:\unit\to\unit(1)$, which is the identity on the underlying one-dimensional vector space, placed in filtration degree 0 (resp.\ 1) in the domain (resp.\ codomain).

  We thank Paul Ziegler for suggesting this example.
\end{Rem}

\section{Simplicity of mixed motives}
\label{sec:DMQ}

Let $\FF$ be a field.
We are interested in classifying motives over $\FF$, and the previous section already provides us with some examples where the category of motives (conjecturally) is simple.
The goal in this section is to discuss these examples and the conjectures involved.

\begin{Rem}[Simplicity of homological motives]
  \label{nori-motives}%
  Let $\FF\subset \CC$ be a field of characteristic zero.
  Nori~\cite{nori-lectures} constructed a Tannakian category of ``homological'' motives $\HM(\FF;\QQ)$ with fiber functor the Betti realization
  \[
    \bti:\HM(\FF;\QQ)\to\modules(\QQ).
  \]
  We conclude from \Cref{tannakian-char0-simple} that $\HM(\FF;\QQ)$ is simple.
\end{Rem}

Let us now turn to the derived setting.
That is, we work with the \emph{triangulated category of mixed motives} (resp.\ \emph{triangulated category of mixed \'etale motives}) $\DMbig(\FF)$ (resp.\ $\DMbiget(\FF)$).
The reader is referred to~\cite[\S\,3, 4]{gallauer:tt-dtm-algclosed} for a brief summary of these theories which contains all that is necessary for the sequel.
In particular, we recall that these two theories coincide with rational coefficients but are different in the presence of torsion.

The following hypothesis will play a key role in the sequel.
\begin{Hyp}
  \label{big-conjecture}%
  There is a field extension $K/\QQ$ and an equivalence of tensor triangulated categories
  \[
    \DM(\FF;K)\simeq \Db(\cA)
  \]
  where $\cA$ is a Tannakian category (necessarily of characteristic zero).
\end{Hyp}

We offer the following remarks regarding this assumption.
In any case, it should be pointed out that \Cref{big-conjecture} is known to imply many motivic conjectures, including in characteristic zero, Grothendieck's Standard conjectures and Beilinson-Soul\'e's Vanishing conjecture.

\begin{Rem}[Comparison with homological motives]
  \label{conjecture-char-0}%
  Assume $\FF\subset\CC$ is a characteristic zero field so that we have the Betti realization $\bti:\DM(\FF;\QQ)\to\Db(\QQ)$ at our disposal.
  We then expect $\cA$ to be a $\QQ$-linear Tannakian category neutralized by the functor induced on hearts, $\bti^{\heartsuit}:\cA\to\modules(\QQ)$.
  
  In fact, a more natural conjecture in this context seems to be the following.
  In addition to the category $\HM(\FF;\QQ)$ of homological motives (\Cref{nori-motives}), Nori constructed a tensor triangulated functor
  \begin{equation}
    \DM(\FF;\QQ)\to\Db(\HM(\FF;\QQ))\label{eq:nori}
  \end{equation}
  compatible with the Betti realizations. (We refer to~\cite[\S\,7]{choudhury-gallauer:iso-galois}, \cite[\S\,10.1]{huber-mueller:periods-nori}, \cite[\S\,7]{harrer:phd} for details.)
  One may then conjecture that this functor~(\ref{eq:nori}) is an equivalence.
\end{Rem}

\begin{Rem}[Ayoub's conjectures]
  \label{ayoub-conjectures}
  A strategy for breaking this conjecture into hopefully easier ones is suggested by Ayoub in terms of (derived) motivic Galois groups~\cite[Introduction]{ayoub:mgg-model}.
  More specifically, assume $\FF\subset\CC$ is of characteristic zero.
  The Betti realization $\bti:\DMbig(\FF;\QQ)\to\D(\QQ)$ admits a right adjoint, and Ayoub constructs an enhancement (on the level of model categories) of the associated Betti monad ${\mathcal B}$.
  It is used to construct an enhanced motivic Hopf algebra ${\mathcal H}_{\FF}$, and Ayoub
  formulates two conjectures~\cite[\S\,2.4]{ayoub:galois1}:
  \begin{enumerate}[label=(\Alph*)]
  \item The homology of $\mathcal{H}_{\FF}$ is concentrated in degree 0.\label{conjecture-A}
  \item For compact motives $M$, the canonical morphism $M\to\holim_n\mathcal{B}^{\circ(n+1)}M$ is an isomorphism.\label{conjecture-B}
  \end{enumerate}
  The two conjectures together should imply that the functor in~(\ref{eq:nori}) is an equivalence.\footnote{More precisely, they do assuming the tedious task of checking that a certain diagram commutes~\cite[p.~1508]{ayoub:mgg-model}.}
\end{Rem}

\begin{Rem}[Ayoub, personal communication]
  Let $\FF$ now be an arbitrary field, and $\ell$ a rational prime invertible in $\FF$.
  Replacing the Betti realization by the $\ell$-adic realization $\ladic:\DMbig(\FF;\QQ_\ell)\to\D(\QQ_\ell)$, one may try to follow the steps in \Cref{ayoub-conjectures} to obtain an $\ell$-adic monad $\mathcal{L}$, construct an associated motivic Hopf algebra $\mathcal{H}_{\FF}$ and formulate the two conjectures analogous to \ref{conjecture-A} and \ref{conjecture-B}.
  Assuming these, one would then hope to obtain a canonical equivalence $\DM(\FF;\QQ_\ell)\isoto\Db(\comod(\Hm_0(\mathcal{H}_{\FF})))$.
\end{Rem}

\begin{Rem}[Non-neutral Tannakian category of motives]
  Assume $\FF$ is a finite field.
  If \Cref{big-conjecture} holds for $\FF$ and $K/\QQ$ then either $K$ is a proper extension, or the Tannakian category $\cA$ is not neutral~\cite[p.~146]{deligne:motifs}.
  We refer to~\cite{milne:finite-fields} for a discussion of (mostly pure) motives in this context.
\end{Rem}

The following observation follows immediately from the discussion in the previous section.
\begin{Thm}
  \label{spc-rational}%
  Assume \Cref{big-conjecture}.
  Then $\DM(\FF;\QQ)$ is simple.
  In particular,
  \begin{itemize}
  \item every non-zero motive generates $\DM(\FF;\QQ)$ as a thick tensor ideal;
  \item every non-zero tensor triangulated functor defined on $\DM(\FF;\QQ)$ is conservative.
  \end{itemize}
\end{Thm}
\begin{proof}
  Let $K/\QQ$ be as in \Cref{big-conjecture}.
  The canonical functor
  \[
    \DM(\FF;\QQ)\to \DM(\FF;K)\simeq\Db(\cA)
  \]
  is faithful and the latter category is simple by \Cref{tannaka-classification}.
  Therefore so is the former, by~\cite[Corollary~1.8]{balmer:surjectivity}.
\end{proof}

\section{Classification of \'etale motives}
\label{sec:tt-etale}

Our next goal is to go beyond rational coefficients, and to say something about torsion phenomena.
In contrast to $\DMbig(\FF;\ZZ)$, there are different useful notions of ``smallness'' for objects in $\DMbiget(\FF;\ZZ)$ (and if the cohomological dimension of $\FF$ is infinite, these are different from being compact).
To fix our ideas we choose to work with the subcategory of rigid objects.
In~\cite{cisinski-deglise:etale-motives} these are called \emph{locally constructible} since they are precisely those motives whose restriction to some finite separable extension $\FF'$ of $\FF$ is constructible of geometric origin (or simply \emph{geometric}, in Voevodsky's terminology), that is,
belongs to the thick subcategory generated by motives of smooth $\FF'$-varieties~\cite[Theorem~6.3.26]{cisinski-deglise:etale-motives}.
The result below remains true (and can in fact be deduced, using~\cite[Corollary~1.8]{balmer:surjectivity}) for the subcategory of geometric motives $\DMgmet(\FF;\ZZ)$.

\begin{Def}
  Let $R$ be a coefficient ring (for us $R=\ZZ$ or a localization or quotient thereof).
  We denote by $\DMlet(\FF;R)$ the full subcategory of $\DMbiget(\FF;R)$ spanned by rigid objects.
  In particular, $\DMlet(\FF;R)$ is a rigid tensor triangulated category.
\end{Def}

Before we can state the main result, we need to recall one more thing.
To any (essentially small) tensor triangulated category $\cT$, one can associate a topological space $\Spc(\cT)$ which consists of the prime thick tensor ideals, endowed with a Zariski type topology~\cite{balmer:spectrum}.
Given an invertible object $u\in\cT$, Balmer~\cite{balmer:sss} also constructs a \emph{comparison map} of topological spaces
\[
  \rho^\sbull_u:\Spc(\cT)\to\Spech(R^\sbull_u)
\]
where $R^\sbull_u=\Hom_{\cT}(\unit,u\potimes{\sbull})$ is the ``graded central ring'' of $\cT$ with respect to $u$, and where $\Spech$ denotes the homogeneous spectrum.
Explicitly, $\rho_u$ sends a prime ideal $\cP$ to the ideal generated by the homogeneous $f\in R^\sbull_u$ such that $\cone(f)\notin\cP$~\cite[Definition~5.1]{balmer:sss}.

\begin{Rem}[Spectrum of Milnor K-theory]
  \label{structure-map-milnor}%
  Assume $\FF$ is of exponential characteristic $p$, that is, $p=1$ if $\chr(\FF)=0$, and $p=\chr(\FF)$ otherwise.
  We now specialize to
  \[
    \cT:=\DMlet(\FF;\ZZ)=\DMlet(\FF;\ZZ[1/p]),\qquad u:=\tilde{\GG}_m=\ZZ(1)[1],
  \]
  the reduced (Tate) motive of $\GG_m$.
  The \'etale motivic cohomology groups in question are known, by deep results of Voevodsky and others.
  Indeed, the Norm Residue Isomorphism (in the form of~\cite[Theorem~6.18]{voevodsky:bloch-kato-conjecture}; see also remarks preceding that result) implies that the diagonal \'etale motivic cohomology groups coincide with the non-\'etale ones:
  \begin{multline*}
    R_{\tilde{\GG}_m}^{\sbull}=\Hom_{\DMlet(\FF;\ZZ[1/p])}(\ZZ,\ZZ(n)[n])\simeq\\ \Hom_{\DM(\FF;\ZZ[1/p])}(\ZZ,\ZZ(n)[n])\simeq \KM_{\sbull}(\FF)[1/p],
  \end{multline*}
  where the second isomorphism is by~\cite{voevodsky:motcoh-CH,nesterenko-suslin:milnor-k-theory,totaro:milnor-k-theory}.

  Finally, we recall that the homogeneous spectrum of Milnor K-theory is known, by an explicit computation in~\cite{thornton:spech-milnor-witt}.
  It has the following points, and is endowed with the minimal topology compatible with the specialization relations (pointing upward):
\[
    \xymatrixcolsep{.05pc}\xymatrix{
 & & & ([\FF^\times],2)  \ar@{-}[ld] \ar@{-}[rd] \ar@{-}[rrrd] & &([\FF^\times],3) \ar@{-}[rd] &([\FF^\times],5) \ar@{-}[d]&\cdots&([\FF^\times],\ell) \\
 & &([P_\alpha],2)& ... &([P_\beta],2)& &([\FF^\times])\ar@{-}[rur]\ar@{}[ru]|{\cdots}}
\]
where $\ell$ runs through the rational primes different from $p$, and where the $P_{\alpha}$ are the orderings of $\FF$ with the Harrison topology (that is, the topology induced by the product topology on $\{\pm 1\}^{\FF^{\times}}$; it is a Boolean space~\cite[VIII.6]{lam:intro-quadratic-forms}).
\end{Rem}

\begin{Thm}
  \label{spc-etale-motives}%
  Let $\FF$ be a field of exponential characteristic $p$ and assume that $\DM(\FF;\QQ)$ is simple (cf.\ \Cref{spc-rational}).
  The comparison map with respect to the Tate motive $\tilde{\GG}_m=\ZZ(1)[1]$ is a homeomorphism
  \begin{equation*}
    \rho^{\sbull}_{\tilde{\GG}_m}:\Spc(\DMlet(\FF;\ZZ))\isoto\Spech(\KM_{\sbull}(\FF)[1/p])
  \end{equation*}
  onto the space described in \Cref{structure-map-milnor}.
\end{Thm}

The proof of \Cref{spc-etale-motives} will proceed along the lines of~\cite[Theorem~6.10]{gallauer:tt-dtm-algclosed}.
We need for this the following generalization of~\cite[Theorem~6.2]{gallauer:tt-dtm-algclosed}.
We learned from Paul Balmer how to remove the assumption that the base field contain a primitive $\ell$th root of unity.
Note also that this Proposition is unconditional (that is, it does not depend on \Cref{big-conjecture}).
\begin{Prop}
  \label{etale-motives-finite-coefficients}
  Let $\FF$ be a field, and let $\ell$ be a prime different from $\car(\FF)$.
  The comparison map with respect to the Tate motive $\tilde{\GG}_m=\ZZ/\ell(1)[1]$ is a homeomorphism
  \[
    \rho_{\tilde{\GG}_m}^{\sbull}:\Spc(\DMlet(\FF;\ZZ/\ell))\isoto \Spech(\KM_{\sbull}(\FF)/\ell).
  \]
\end{Prop}

  \begin{proof}
    By Rigidity (Suslin-Voevodsky, \cite[Theorem~6.3.11]{cisinski-deglise:etale-motives}), the tensor triangulated category $\DMlet(\FF;\ZZ/\ell)$ is identified with $\Dbc(\Mod(G_{\FF};\ZZ/\ell))$, the bounded constructible derived category of discrete $G_{\FF}$-modules over $\ZZ/\ell$, where $G_{\FF}$ denotes the absolute Galois group of $\FF$.
    Then \cite[Lemma~6.4]{gallauer:tt-dtm-algclosed} provides a further identification with $\Db(\modules(G_{\FF};\ZZ/\ell))$, the bounded derived category of finite-dimensional discrete $G_{\FF}$-modules.
    It follows that
    \[
      R^{\sbull}_{\tilde{\GG}_m}\simeq  \Hm^{\sbull}(G_{\FF},\mu_{\ell}\potimes{\sbull})\simeq K_{\sbull}^M(\FF)/\ell
    \]
    by the Norm Residue Isomorphism (due to Voevodsky \etal).

    It remains to show that the comparison map $\rho^{\sbull}_{\tilde{\GG}_m}$ is a homeomorphism.
    Assume first that $\FF$ contains a primitive $\ell$th root of unity.
    In that case, $\tilde{\GG}_m=\ZZ/\ell(1)[1]\simeq\ZZ/\ell[1]$ and we have proven in~\cite[Proposition~6.5]{gallauer:tt-dtm-algclosed} that the comparison map $\rho^{\sbull}_{\ZZ/\ell[1]}$ is a homeomorphism.

    If $\FF$ does not contain a primitive $\ell$th root of unity, let $\FF':=\FF[\zeta]$ be obtained by adjoining such.
    Necessarily, $\ell$ is odd.
    As recalled in \Cref{structure-map-milnor}, the homogeneous spectrum of Milnor $K$-theory at odd primes is a singleton space, for any field.
    We are therefore reduced to show that $\Spc(\DMlet(\FF;\ZZ/\ell))=\ast$ is a singleton space too.
    But, $\ell$ being odd, we deduce that the degree of the field extension $\FF'$ is $[\FF':\FF]=\ell-1$ which is invertible in $\ZZ/\ell$.
    The composite of the base change functor $\DMlet(\FF;\ZZ/\ell)\to\DMlet(\FF';\ZZ/\ell)$ with its right adjoint is tensoring with the motive of $\Spec(\FF')$ which is faithful (since the Euler characteristic of this motive is $\ell-1$).
    It follows that the base change functor itself is faithful, and so the induced map on spectra is surjective~\cite[Corollary~1.8]{balmer:surjectivity}.
    Since $\Spc(\DMlet(\FF';\ZZ/\ell))=\ast$, we deduce that the same is true for $\FF$, and this concludes the argument.
    
\end{proof}

  \begin{proof}[Proof of \Cref{spc-etale-motives}]
    We start by showing that $\rho^{\sbull}:=\rho_{\tilde{\GG}_m}^{\sbull}$ is a bijection.
  For this we think of the map as fibered over $\Spec(\ZZ[1/p])$ and analyze it one prime $(\ell)\neq (p)$ at a time.

  If $\ell=0$, we use~\cite[Corollary~5.6(c)]{balmer:sss} to identify the fiber of $\rho^{\sbull}$ over $\Spec(\QQ)$ with
  \[
    \Spc(\DMlet(\FF;\ZZ)\otimes\QQ)\xrightarrow{\rho^{\sbull}}\Spech(K^{M}_{\sbull}(\FF)\otimes\QQ).
  \]
  The codomain is a singleton space, and so is the domain by our assumption, given the equivalences~(\cite[Theorem~16.1.2]{cisinski-deglise:dm} and~\cite[Corollary~5.4.9]{cisinski-deglise:etale-motives}):
  \[
    \left(
      \DMlet(\FF;\ZZ)\otimes\QQ
\right)^{\natural}\simeq\DMlet(\FF;\QQ)\simeq\DM(\FF;\QQ).
  \]
  (Here, $(-)^{\natural}$ denotes the idempotent completion, and we also use~\cite[Proposition~3.13]{balmer:spectrum}.)
  
  Next, if $\ell$ is a finite prime, the fiber of
  \[
    \rho:\Spc(\DMlet(\FF;\ZZ))\xrightarrow{\rho^{\sbull}}\Spech(K^{M}_{\sbull}(\FF))\to\Spec(\ZZ)
  \]
  over $\ell$ is by definition the support of $\ZZ/\ell$ in both the tensor triangular spectrum and the homogeneous spectrum.
  Consider the quotient morphism $\gamma:\ZZ\to\ZZ/\ell$, and the induced adjunction on categories of \'etale motives:
  \[
    \gamma^*:\DMlet(\FF;\ZZ)\rightleftarrows\DMlet(\FF;\ZZ/\ell):\gamma_*
  \]
  We have a commutative square~\cite[Theorem~5.3(c)]{balmer:sss}:
  \begin{equation}
    \label{eq:support-Zl-square}
    \xymatrix{
      \Spc(\DMlet(\FF;\ZZ/\ell))
      \ar[r]^{\rho^{\sbull}_{\tilde{\GG}_m}}_{\sim}
      \ar[d]_{\Spc(\gamma^*)}
      &
      \Spech(K_{\sbull}^M(\FF)/\ell)
      \ar[d]^{\Spech(\gamma)}
      \\
      \Spc(\DMlet(\FF;\ZZ))
      \ar[r]_{\rho^{\sbull}}
      &\Spech(K_{\sbull}^M(\FF))}
  \end{equation}
  where the top horizontal arrow is a homeomorphism by \Cref{etale-motives-finite-coefficients}.
  Since $\ZZ/\ell$ is the image of the unit under $\gamma_*$, we deduce from \cite[Theorem~1.7]{balmer:surjectivity} that $\supp(\ZZ/\ell)=\Img(\Spc(\gamma^*))$.
  Also, the right vertical arrow is injective with image the support of $\ZZ/\ell$.
  It follows that the left vertical arrow is injective as well, and the bottom horizontal map $\rho^{\sbull}$ indeed identifies the support of $\ZZ/\ell$ on both sides, as required.

  We will now show that this map
  \[
    \rho^{\sbull}:\Spc(\DMlet(\FF;\ZZ))\to\Spech(K_{\sbull}^M(\FF))
  \]
  is a homeomorphism.
  First, note that in the square~\eqref{eq:support-Zl-square}, the left vertical arrow is a homeomorphism onto its image.
  (In fact, the square~\eqref{eq:support-Zl-square} is cartesian.)
  This is relevant for $\ell=2$ as it reduces us now to prove that the specialization relations $([\FF^\times])\rightsquigarrow ([\FF^\times],\ell)$ lift to specialization relations in $\Spc(\DMlet(\FF;\ZZ))$.
  (We use that a bijective spectral map between spectral spaces which lifts all specialization relations is necessarily a homeomorphism.)
  In other words, we need to show the inclusion of prime ideals
  \[
    \ker(-\otimes\ZZ/\ell)\subset\ker(-\otimes\QQ)
  \]
  in the category $\DMlet(\FF;\ZZ_{\langle\ell\rangle})$.
  (Here, $\ZZ_{\langle\ell\rangle}\subset\QQ$ is the localization of $\ZZ$ at the prime ideal generated by $\ell$.)

  For this we use the $\ell$-adic completion (interpreted as the $\ell$-adic realization) discussed in~\cite[\S\,7.2]{cisinski-deglise:etale-motives}, with $R$ in their notation being our $\ZZ_{\langle\ell\rangle}$.
  Suppose that $M\in\DMlet(\FF,\ZZ_{\langle{\ell}\rangle})$ satisfies $M\otimes\ZZ/\ell\simeq 0$.
  It follows from~\cite[Lemma~7.2.3 and (7.2.4.a)]{cisinski-deglise:etale-motives} that the $\ell$-adic realization $\hat{\rho}_\ell^*M\simeq 0$ of $M$ vanishes as well.
  And so does therefore the image of $M$ under the composite of the left vertical and bottom horizontal arrow in the following diagram:
  \[
    \xymatrix{
      \DMlet(\FF;\ZZ_{\langle\ell\rangle})
      \ar[d]^{\hat{\rho}_{\ell}^*}
      \ar[r]^{\otimes\QQ}
      &
      \DMlet(\FF;\QQ)
      \ar[d]^{\hat{\rho}_{\ell}^*\otimes\QQ}
      \\
      \Dbc(\FF;\ZZ_{\ell})
      \ar[r]^{\otimes\QQ}
      &
      \Dbc(\FF;\QQ_{\ell})
      }
  \]
  The $\ell$-adic realization on the right is defined as the $\QQ$-linearization of $\hat{\rho}^*_\ell$ and is an exact tensor functor (\cite[Theorem~7.2.24]{cisinski-deglise:etale-motives}).
  In particular, the square commutes and we deduce that $M\otimes\QQ$ lies in the kernel of $\ell$-adic realization.
  By our assumption, the category $\DMlet(\FF;\QQ)$ is simple and this kernel is therefore trivial.
  We conclude that $M\otimes\QQ\simeq 0$ as required.
  (Of course, we then necessarily have $M\simeq 0$ to start with.)
\end{proof}

\begin{Rem}[Description of prime ideals]
  \label{primes}%
  Let us describe the primes in $\DMlet(\FF;\ZZ)$ more explicitly.
  \begin{enumerate}
  \item The prime corresponding to $([\FF^\times])$ consists of the \emph{torsion objects}, that is, those objects $M$ with $M\otimes\QQ\simeq 0$.
    Since (under the assumption that \Cref{big-conjecture} is satisfied) any of the standard cohomology theories is conservative for motives with rational coefficients, this prime is also equal to the kernel of the functor
    \[
      \DMlet(\FF;\ZZ)\xrightarrow{\ladic} \Db(\QQ_{\ell})
    \]
    for any $\ell\neq p$.
  \item The prime corresponding to $([\FF^\times],\ell)$ for $\ell\neq p$ consists of the objects $M$ with $M\otimes \ZZ/\ell\simeq 0$.
    By Rigidity (see the proof of \Cref{etale-motives-finite-coefficients}), it can also be described as the kernel of the functor
    \[
      \DMlet(\FF;\ZZ)\xrightarrow{\times_{\FF}\bar{\FF}}\DMlet(\bar{\FF};\ZZ)\xrightarrow{/\ell} \DMlet(\bar{\FF};\ZZ/\ell)\simeq\Db(\ZZ/\ell).
    \]
  \item Finally, let $P_\alpha$ be a positive cone of $\FF$, and choose a completion $\FF_{\alpha}$.
    Then the prime corresponding to $([P_\alpha],2)$ is the kernel of the functor
    \begin{multline*}
      \DMlet(\FF;\ZZ)\xrightarrow{\times_{\FF}\FF_\alpha}\DMlet(\FF_\alpha;\ZZ)\xrightarrow{/2}      \DMlet(\FF_\alpha;\ZZ/2)\simeq\\
 \Db(\ZZ/2[C_2])\onto\stab(\ZZ/2[C_2])\simeq \modules(\ZZ/2),
    \end{multline*}
    where $C_2=\Gal(\bar{\FF}_\alpha/\FF_\alpha)$ denotes the cyclic group of order 2, and the functor $\Db(\ZZ/2[C_2])\onto\stab(\ZZ/2[C_2])$ is quotienting out the complexes of projective objects.
  \end{enumerate}
\end{Rem}
\bibliographystyle{alpha}
\bibliography{ref-ttdmet}

\begin{thebibliography}{CGAdS17}

\bibitem[Ayo14]{ayoub:galois1}
Joseph Ayoub.
\newblock L'alg\`ebre de {H}opf et le groupe de {G}alois motiviques d'un corps
  de caract\'eristique nulle, {I}.
\newblock {\em J. Reine Angew. Math.}, 693:1--149, 2014.

\bibitem[Ayo17]{ayoub:mgg-model}
Joseph Ayoub.
\newblock From motives to comodules over the motivic {H}opf algebra.
\newblock {\em J. Pure Appl. Algebra}, 221(7):1507--1559, 2017.

\bibitem[Bal05]{balmer:spectrum}
Paul Balmer.
\newblock The spectrum of prime ideals in tensor triangulated categories.
\newblock {\em J. Reine Angew. Math.}, 588:149--168, 2005.

\bibitem[Bal10]{balmer:sss}
Paul Balmer.
\newblock Spectra, spectra, spectra - tensor triangular spectra versus
  {Z}ariski spectra of endomorphism rings.
\newblock {\em Algebraic and Geometric Topology}, 10(3):1521--63, 2010.

\bibitem[Bal18]{balmer:surjectivity}
Paul Balmer.
\newblock On the surjectivity of the map of spectra associated to a
  tensor-triangulated functor.
\newblock {\em Bull. Lond. Math. Soc.}, 50(3):487--495, 2018.

\bibitem[BG]{balmer-gallauer:artin-motives}
Paul Balmer and Martin Gallauer.
\newblock In preparation.

\bibitem[BG19]{balmer-gallauer:tt-ratm}
Paul {Balmer} and Martin {Gallauer}.
\newblock {Three real Artin-Tate motives}.
\newblock {\em arXiv e-prints}, Jun 2019.

\bibitem[BKS18]{balmer-krause-stevenson:frame-smashing}
Paul Balmer, Henning Krause, and Greg Stevenson.
\newblock The frame of smashing tensor-ideals.
\newblock {\em Mathematical Proceedings of the Cambridge Philosophical
  Society}, page 1–21, 2018.

\bibitem[CD16]{cisinski-deglise:etale-motives}
Denis-Charles Cisinski and Fr{\'e}d{\'e}ric D{\'e}glise.
\newblock {\' E}tale motives.
\newblock {\em Compositio Mathematica}, 152(3):556--666, 003 2016.

\bibitem[CD19]{cisinski-deglise:dm}
Denis-Charles Cisinski and Fr{\'e}d{\'e}ric D{\'e}glise.
\newblock {\em {Triangulated categories of mixed motives}}.
\newblock Springer Monographs in Mathematics. Springer International
  Publishing, 2019.

\bibitem[CGAdS17]{choudhury-gallauer:iso-galois}
Utsav Choudhury and Martin Gallauer Alves~de Souza.
\newblock An isomorphism of motivic {G}alois groups.
\newblock {\em Adv. Math.}, 313:470--536, 2017.

\bibitem[Del90]{deligne:cat-tann}
P.~Deligne.
\newblock Cat\'{e}gories tannakiennes.
\newblock In {\em The {G}rothendieck {F}estschrift, {V}ol. {II}}, volume~87 of
  {\em Progr. Math.}, pages 111--195. Birkh\"{a}user Boston, Boston, MA, 1990.

\bibitem[Del94]{deligne:motifs}
Pierre Deligne.
\newblock \`a quoi servent les motifs?
\newblock In {\em Motives ({S}eattle, {WA}, 1991)}, volume~55 of {\em Proc.
  Sympos. Pure Math.}, pages 143--161. Amer. Math. Soc., Providence, RI, 1994.

\bibitem[DG13]{MR3037900}
Vladimir Drinfeld and Dennis Gaitsgory.
\newblock On some finiteness questions for algebraic stacks.
\newblock {\em Geom. Funct. Anal.}, 23(1):149--294, 2013.

\bibitem[Gal18]{gallauer:tt-fmod}
Martin Gallauer.
\newblock Tensor triangular geometry of filtered modules.
\newblock {\em Algebra Number Theory}, 12(8):1975--2003, 2018.

\bibitem[Gal19]{gallauer:tt-dtm-algclosed}
Martin Gallauer.
\newblock tt-geometry of {T}ate motives over algebraically closed fields.
\newblock {\em Compositio Mathematica}, 155(10):1888–1923, 2019.

\bibitem[Har16]{harrer:phd}
Daniel Harrer.
\newblock {\em Comparison of the Categories of Motives defined by Voevodsky and
  Nori}.
\newblock PhD thesis, Albert-Ludwigs-Universit\"{a}t Freiburg im Breisgau,
  2016.

\bibitem[HMS17]{huber-mueller:periods-nori}
Annette Huber and Stefan M\"uller-Stach.
\newblock {\em Periods and {N}ori motives}, volume~65 of {\em Ergebnisse der
  Mathematik und ihrer Grenzgebiete. 3. Folge. A Series of Modern Surveys in
  Mathematics [Results in Mathematics and Related Areas. 3rd Series. A Series
  of Modern Surveys in Mathematics]}.
\newblock Springer, Cham, 2017.
\newblock With contributions by Benjamin Friedrich and Jonas von Wangenheim.

\bibitem[HR15]{hall-rydh:compact-generation}
Jack Hall and David Rydh.
\newblock Algebraic groups and compact generation of their derived categories
  of representations.
\newblock {\em Indiana Univ. Math. J.}, 64(6):1903--1923, 2015.

\bibitem[HR17]{hall-rydh:telescope-stacks}
Jack Hall and David Rydh.
\newblock The telescope conjecture for algebraic stacks.
\newblock {\em Journal of Topology}, 10(3):776--794, 2017.

\bibitem[Kel96]{keller:derived-categories}
Bernhard Keller.
\newblock Derived categories and their uses.
\newblock In {\em Handbook of algebra}, volume~1, pages 671--701.
  Elsevier/North-Holland, Amsterdam, 1996.

\bibitem[Lam05]{lam:intro-quadratic-forms}
T.~Y. Lam.
\newblock {\em Introduction to quadratic forms over fields}, volume~67 of {\em
  Graduate Studies in Mathematics}.
\newblock American Mathematical Society, Providence, RI, 2005.

\bibitem[Mil94]{milne:finite-fields}
J.~S. Milne.
\newblock Motives over finite fields.
\newblock In {\em Motives ({S}eattle, {WA}, 1991)}, volume~55 of {\em Proc.
  Sympos. Pure Math.}, pages 401--459. Amer. Math. Soc., Providence, RI, 1994.

\bibitem[Nor]{nori-lectures}
Madhav Nori.
\newblock Lectures at {TIFR}.
\newblock 32 pages.

\bibitem[NS89]{nesterenko-suslin:milnor-k-theory}
Yu.~P. Nesterenko and A.~A. Suslin.
\newblock Homology of the general linear group over a local ring, and
  {M}ilnor's {$K$}-theory.
\newblock {\em Izv. Akad. Nauk SSSR Ser. Mat.}, 53(1):121--146, 1989.

\bibitem[Pet13]{peter:spectrum-damt}
Tobias~J. Peter.
\newblock Prime ideals of mixed {A}rtin-{T}ate motives.
\newblock {\em Journal of K-Theory}, 11(2):331--349, 004 2013.

\bibitem[Tho16]{thornton:spech-milnor-witt}
Riley Thornton.
\newblock The homogeneous spectrum of {M}ilnor-{W}itt {$K$}-theory.
\newblock {\em J. Algebra}, 459:376--388, 2016.

\bibitem[Tot92]{totaro:milnor-k-theory}
Burt Totaro.
\newblock Milnor {$K$}-theory is the simplest part of algebraic {$K$}-theory.
\newblock {\em $K$-Theory}, 6(2):177--189, 1992.

\bibitem[Voe02]{voevodsky:motcoh-CH}
Vladimir Voevodsky.
\newblock Motivic cohomology groups are isomorphic to higher {C}how groups in
  any characteristic.
\newblock {\em Int. Math. Res. Not.}, (7):351--355, 2002.

\bibitem[Voe11]{voevodsky:bloch-kato-conjecture}
Vladimir Voevodsky.
\newblock On motivic cohomology with $\mathbb{Z}/l$-coefficients.
\newblock {\em Ann. of Math. (2)}, 174(1):401--438, 2011.

\end{thebibliography}
\end{document}